\documentclass{amsart}
\usepackage{amsmath, amsthm, amsfonts, amssymb}
\usepackage{mathrsfs}
\usepackage{microtype}
\usepackage[utf8]{inputenc}
\providecommand{\noopsort[1]{}}
\usepackage{bbm}
\usepackage{dsfont}
\usepackage{ifthen}
\usepackage{nicefrac}
\numberwithin{equation}{section}
\usepackage{a4}
\usepackage[active]{srcltx}
\usepackage{verbatim}
\usepackage{hyperref}
\usepackage{color}

\usepackage[shortlabels]{enumitem}
\setlist{leftmargin=*}
\setlist[1]{labelindent=1.2\parindent}

\allowdisplaybreaks

\newtheorem{thm}{Theorem}[section]
\newtheorem{coro}[thm]{Corollary}
\newtheorem{prop}[thm]{Proposition}
\newtheorem{lm}[thm]{Lemma}

\theoremstyle{remark}
\newtheorem{rmk}[thm]{Remark}

\newtheorem{hyp}[thm]{Hypotheses}
\newtheorem{examp}[thm]{Example}

\newcommand{\coloneqq}{\mathrel{\mathop:}=}

\renewcommand{\Re}{{\rm Re}\,}

\renewcommand{\Im}{{\rm Im}\,}
\newcommand{\eps}{\varepsilon}

\newcommand{\loc}{\mathrm{loc}}
\newcommand{\R}{\mathds{R}}
\newcommand{\C}{\mathds{C}}

\newcommand{\N}{\mathds{N}}

\newcommand{\tnorm}[1]{{|\kern-0.3ex|\kern-0.3ex|#1|\kern-0.3ex|\kern-0.3ex|}}

\newcommand{\la}{\langle}
\newcommand{\ra}{\rangle}

\DeclareMathOperator{\rg}{\mathrm{rg}}

\let\div\undefined
\DeclareMathOperator{\div}{\mathrm{div}}

\begin{document}
\title{Vector-valued  Schr\"odinger operators in $L^p$-spaces}
\author{M. Kunze}
\address{Universit\"at Konstanz, Fachbereich Mathematik und Statistik, 78457 Konstanz, Germany}
\email{markus.kunze@uni-konstanz.de}
\author{A. Maichine}
\address{Dipartimento di Matematica, Universit\`a degli Studi di Salerno, Via Giovanni Paolo II 132, I-84084 Fisciano (SA), Italy}
\email{amaichine@unisa.it}
\author{A. Rhandi}
\address{Dipartimento di Ingegneria dell'Informazione, Ingegneria Elettrica e Matematica Applicata, Università degli Studi di Salerno, Via Ponte Don Melillo 1, 84084 Fisciano (Sa), Italy}
\thanks{This work has been supported by the M.I.U.R. research project
Prin 2015233N54 "Deterministic and Stochastic Evolution Equations". The second author is a member of the Gruppo Nazionale per l'Analisi Matematica, la Probabilit\`{a} e le loro Applicazioni
(GNAMPA) of the Istituto Nazionale di Alta Matematica (INdAM).}
\address{Dipartimento di Ingegneria dell'Informazione, Ingegneria Elettrica e Matematica Applicata, Università degli Studi di Salerno, Via Ponte Don Melillo 1, 84084 Fisciano (Sa), Italy}
\email{arhandi@unisa.it}
\begin{abstract}
In this paper we consider vector-valued Schr\"odinger operators of the form $\div(Q\nabla u)-Vu$, where $V=(v_{ij})$ is a nonnegative locally bounded matrix-valued function and $Q$ is a symmetric, strictly elliptic matrix whose entries are bounded and continuously differentiable with bounded derivatives.
Concerning the potential $V$, we assume an that it is pointwise accretive and that its entries are in $L^\infty_{\loc}(\R^d)$.
Under these assumptions, we prove that a realization of the vector-valued Schr\"odinger operator generates a $C_0$-semigroup of contractions in $L^p(\R^d; \C^m)$. Further properties are also investigated.
\end{abstract}
\maketitle
\section{Introduction}

Recently, there is an increased interest in systems of parabolic equations with unbounded coefficients. Such systems appear in the study of backward-forward differential games, in connection with Nash equilibria in stochastic differential games, in the analysis of the weighted $\overline{\partial}$-problem in $\C^d$, in time dependent Born--Openheimer theory and also in the study of Navier--Stokes equation. For more information we refer the reader to
\cite[Section 6]{aalt}, \cite{Ha-He}, \cite{Dall}, \cite{BGT}, \cite{hieber}, \cite{HRS} and \cite{Ha-Rh}.

While the scalar theory of such equations is by now well understood (see \cite{lb07} and the references therein),  so far there are only few articles concerned with systems. We mention the article \cite{hetal09} where systems of parabolic equations coupled through both, a potential term and a drift term, were considered in the $L^p$-setting.
In \cite{aalt, alp16, dl11}, the authors choose a different approach. Indeed, they first constructed solutions in the space of bounded and continuous functions and only afterwards the obtained semigroup is extrapolated to the $L^p$-scale. We should point out
that in the presence of an unbounded drift term the differential operator does not always generate a strongly continuous semigroup on $L^p$-spaces with respect to Lebesgue measure, see \cite{prs06}. Thus, in some cases appropriate growth conditions need to be imposed on the coefficients to ensure generation of a semigroup on $L^p$ with respect to Lebesgue measure.

In this article we will consider systems of parabolic equations which are coupled only through a potential term.
To be more precisely, consider the differential operator
\begin{equation}\label{eq. mathcal A=}
\mathcal{A}u=\div(Q\nabla u)-Vu=\colon\Delta_Q u-Vu
\end{equation}
acting on vector-valued functions $u=(u_1,\dots,u_m)\colon \R^d\to\C^m$. Here $Q$ is a bounded, symmetric and stricly elliptic matrix with continuously differentiable entries that have bounded derivatives. The expression $\div (Q\nabla u)$
should be understood componentwise, i.e.\ $\div (Q\nabla u ) = (\div (Q\nabla u_1), \ldots, \div (Q\nabla u_m))$.
The matrix-valued function $V: \R^d \to \R^{m\times m}$ is assumed to be pointwise accretive and to have locally bounded coefficients. In contrast to the situation where an unbounded drift is present, no additional growth assumptions on the potential $V$ are needed to ensure generation of a strongly continuous semigroup on $L^2(\R^d; \R^m)$.
Indeed, following Kato \cite{Kato78}, who considered the scalar situation, we shall construct a densely defined,
m-dissipative realization $A$ of the operator $\mathcal{A}$ in $L^2(\R^d;\R^m)$. By virtue of the Lumer--Phillips theorem $A$ generates a strongly continuous semigroup. Subsequently, we prove that this semigroups extrapolates to
a consistent family of strongly continuous contraction semigroups $\{T_p(t)\}_{t\geq 0}$ on $L^p(\R^d; \R^m)$
for $1<p<\infty$. We also give a description of the generator $A_p$ of $\{T_p(t)\}_{t\geq 0}$ and prove that the test functions form a core for the operator $A_p$.

We should point out that in our recent article \cite{KLMR} we were studying a similar setting. However, in \cite{KLMR} we were interested in proving that the domain of the vector-valued Schr\"odinger operator is the intersection of the domain of the diffusion part and the potential part. To that end, we had to impose growth conditions on the Potential part. Here, we allow general potential without such a growth condition. The price to pay is that we can only characterize the domain of the $L^p$-realization of our operator as the maximal $L^p$-domain.

This article is organised as follows. In Section \ref{s.kato} we prove a version of Kato's inequality for vector-valued functions which is crucial in all subsequent sections. In Section \ref{sec. gen of sg in L2} we construct a realization
of the operator $\mathcal{A}$ in $L^2(\R^d;\R^m)$ which generates a strongly continuous contraction semigroup.
In Section \ref{Extension to Lp}, we extrapolate the semigroup to $L^p$-spaces, where $p\in (1,\infty)$.
In the concluding Section \ref{sec. maximal domain} we characterize the domain of the generator as maximal domain.

\textbf{Notation} Let $d,m\geq 1$. By $|\cdot |$ we denote the Euclidean norm on $\C^j$, $j=d,m$ and by $\langle \cdot,\cdot \rangle$ the Euclidean inner product. By $B(r)=\{x\in\R^d : |x|\leq r \}$ we denote the Euclidean ball of radius $r>0$ and center $0$. For $1\leq p\leq \infty$, $L^{p}(\R^d;\C^m)$ is the $\C^m$- valued Lebesgue space on $\R^d$. For $1\leq p< \infty$, the norm is given by
\[ \|f\|_p \coloneqq \left(\int_{\R^d} |f(x)|^p dx \right)^\frac{1}{p}=\left(\int_{\R^d} \Big(\sum_{j=1}^{m}|f_j|^2\Big)^{\frac{p}{2}} dx \right)^\frac{1}{p}, \quad f\in
 L^{p}(\R^d,\C^m),\]
 whereas in the case $p=\infty$ we use the essential supremum norm
 \[
 \|f\|_\infty \coloneqq \operatorname{ess} \sup \{ |f(x)| : x\in \R^d\}.
 \]
For $1<p<\infty$, $p'$ refers to the conjugate index, i.e. $1/p+1/p'=1$. Thus $L^{p'}(\R^d;\C^m)$ is the dual space of $L^{p}(\R^d,\C^m)$ and the duality pairing $\langle \cdot,\cdot \rangle_{p,p'}$ is given by
\[ \langle f,g\rangle_{p,p'}=\int_{\R^d}\langle f(x),g(x)\rangle dx,\qquad \mbox{for} f\in L^{p}(\R^d;\C^m), g \in L^{p'}(\R^d;\C^m). \]
By $C_c^\infty (\R^d;\C^m)$, we denote the space of all test functions, i.e.\ functions $f:\R^d\to \C^m$ which have compact support and derivatives of any order. The space $W^{k,p}(\R^d;\C^m)$ is the classical Sobolev space of order $k$, that is the space of all functions  $f\in L^{p}(\R^d;\C^m)$ such that the distributional derivative $\partial^\alpha f$ belongs to $L^{p}(\R^d;\C^m)$ for all $\alpha=(\alpha_1,\dots,\alpha_d)\in\N^d$ with $|\alpha|=\sum_{j=1}^{d}\alpha_j\leq k$. For $1\leq p \leq \infty$ we define $L^p_\loc(\R^d;R^m)$ as the space of all
measurable functions $f: \R^d\to \R^m)$ such that $\chi_{B(r)}f \in L^p(\R^d;\R^m)$ for all $r>0$. Here $\chi_{B(r)}$ is the indicator function of the ball $B(r)$. The space $W^{k,p}_{\loc} (\R^d; \R^m)$ is the space of all functions $f\in L^p_\loc (\R^d;\R^m)$ such that for $\alpha \in \N^d$ with $|\alpha|\leq k$ the distributional derivative
$\partial^\alpha f$ belongs to $L^p_\loc (\R^d;\R^m)$. We write $H^k(\R^d;\C^m) \coloneqq W^{k,2}(\R^d;\C^m)$ and
$H^k_\loc (\R^d;\C^m) \coloneqq W^{k,2}_\loc (\R^d;\C^m)$.

\section{Preliminaries}\label{s.kato}
Troughout this article we make the following assumptions:
\begin{hyp}\label{Hyp}
\begin{enumerate}
\item The map $Q:\R^d\to\R^{d\times d}$ is such that $q_{ij}=q_{ji}$ is bounded and continuously differentiable with bounded derivative for all $i,j\in\{1,\dots, d\}$ and there exist positive real numbers $\eta_1$ and $\eta_2$ such that
\begin{equation}\label{ellip of Q}
\eta_1|\xi|^2\le \la Q(x)\xi,\xi\ra\le \eta_2|\xi|^2
\end{equation}
for all $x,\xi\in\R^d$.
\item The map $V:\R^d\to\R^{m\times m}$ has entries $v_{ij}\in L^\infty_\loc(\R^d)$ for all $i,j\in\{1,\dots,m\}$ and
\begin{equation}\label{eq.diss}
\Re \la V(x)\xi,\xi\ra\ge 0,
\end{equation}
for all $x\in\R^d$, $\xi\in\C^m$.
\end{enumerate}
\end{hyp}
To simplify  notations, we write for $\xi, \eta \in \R^d$
\[
\langle \xi, \eta\rangle_Q \coloneqq \sum_{i,j=1}^d q_{ij}\xi_i\eta_j\quad\mbox{and}\quad
|\xi|_Q \coloneqq \sqrt{\langle \xi, \xi\rangle_Q}.
\]
 We define the operator $\Delta_Q : W^{1,1}_\loc (\R^d)  \to \mathscr{D}(\R^d)$ by setting
\begin{equation}\label{eq.deltaq}
\langle \Delta_Q u, \varphi\rangle = - \int_{\R^d} \langle \nabla u, \nabla \varphi\rangle_Q\, dx.
\end{equation}
for any test function $\varphi \in C_c^\infty(\R^d)$, where $\mathscr{D}(\R^d)$ denotes the space of distributions. As usual, we will say that $\Delta_Qu \in L^1_\loc(\R^d)$, if
there is a function $f\in L^1_\loc(\R^d)$ such that
\[
\langle \Delta_Qu, \varphi\rangle = \int_{\R^d} f\varphi\, dx
\]
for all $\varphi \in C_c^\infty(\R^d)$. In this case we will identify $\Delta_Q u$ and the function $f$.

The following lemma, taken from \cite[Lemma 2.4]{Mai-Rha},  generalizes Stampacchia's result concerning the weak derivativ  of the absolute value of an $W^{1,p}$-function, see \cite[Lemma 7.6]{GilTru}, to vector-valued functions..
\begin{lm}\label{l.normsobolev}
 Let $1<p<\infty$ and $u=(u_1(x),\dots,u_m(x))\in W^{1,p}_\loc(\R^d;\C^m)$. Then,  $|u|\in W^{1,p}_\loc (\R^d)$ and
 \begin{equation}\label{gradient of |.|=}
 \nabla|u|=\frac{1}{|u|}\sum_{j=1}^{m}\Re(\bar{u}_j\nabla u_j)\chi_{\{u\ne 0\}}.
 \end{equation}
 Moreover,
 \begin{equation}\label{gradient |.| <...}
 |\nabla |u||_Q^2 \leq\sum_{j=1}^{d}|\nabla u_j|_Q^2.
 \end{equation}
 \end{lm}

We can now prove a vector-valued version of Kato's inequality.
\begin{prop}\label{p.kato}
Let $u = (u_1, \ldots, u_m) \in H^1_\loc (\R^d;\R^m)$ be such that $\Delta_Qu_j \in L^1_\loc (\R^d)$ for
$j=1, \ldots, m$. Then
\begin{equation}\label{eq.DeltaQ}
\Delta_Q |u| = \chi_{\{u\neq 0\}} \frac{1}{|u|}\Big( \sum_{j=1}^m u_j \Delta_Qu_j + \sum_{j=1}^m |\nabla u_j|_Q^2 - |\nabla |u||_Q^2\Big).
\end{equation}
Thus, the Kato inequality
\begin{equation}\label{eq.kato}
\Delta_Q|u| \geq \chi_{\{u\neq 0\}} \frac{1}{|u|}\sum_{j=1}^m u_j \Delta_Qu_j
\end{equation}
holds in the sense of distributions.
\end{prop}

\begin{proof}
Let us consider the function $a_\eps(u)= \big(|u|^2+\eps^2\big)^\frac{1}{2} -\eps$. It was proved in \cite[Lemma 2.4]{Mai-Rha} that $\lim_{\eps \to 0}a_\eps(u) = |u|$ in $H^1_{\loc}(\R^d)$. This implies that $\lim_{\eps \to 0}\Delta_Q a_\eps(u) = \Delta_Q |u|$ in $\mathscr{D}(\R^d)$.

Since $\nabla a_\eps(u) = \frac{1}{a_\eps(u)+\eps}\sum_{j=1}^m u_j \nabla u_j$, it follows that
for $ \varphi \in C_c^\infty(\R^d)$ we have
\begin{align*}
\langle \Delta_Q a_\eps(u), \varphi\rangle
  &=  -\int_{\R^d} \langle Q\nabla a_\eps(u), \nabla\varphi\rangle\, dx
= - \sum_{j=1}^m \int_{\R^d} \frac{u_j}{a_\eps(u)+\eps}\langle Q\nabla u_j, \nabla \varphi\rangle\, dx\\
 & = - \sum_{j=1}^m \int_{\R^d} \big\langle Q\nabla u_j, \nabla ((a_\eps(u)+\eps)^{-1}u_j \varphi )\big\rangle\, dx\\
 &\qquad
+ \sum_{j=1}^m \int_{\R^d} \big\langle Q\nabla u_j, \nabla ((a_\eps(u)+\eps)^{-1}u_j)\big\rangle\varphi \, dx\\
 &= \sum_{j=1}^m\int_{\R^d} \frac{u_j}{a_\eps(u)+\eps}\Delta_Q u_j \varphi\, dx
+ \sum_{j=1}^m \int_{\R^d} \frac{1}{a_\eps(u)+\eps} \langle  Q\nabla u_j ,\nabla u_j \rangle \varphi \, dx\\
&\qquad
- \sum_{j=1}^m \int_{\R^d} \frac{u_j}{(a_\eps(u)+\eps)^2} \langle Q \nabla u_j ,\nabla a_\eps(u) \rangle\varphi \, dx\\
 &= \sum_{j=1}^m\int_{\R^d} \frac{u_j}{a_\eps(u)+\eps}\Delta_Q u_j \varphi\, dx
+ \sum_{j=1}^m \int_{\R^d} \frac{1}{a_\eps(u)+\eps} \langle  Q\nabla u_j ,\nabla u_j \rangle \varphi \, dx\\
& \qquad
- \int_{\R^d} \frac{1}{a_\eps(u)+\eps}\langle Q\nabla a_\eps(u), \nabla a_\eps(u) \rangle\varphi\, dx.
\end{align*}
Recall that $\lim_{\eps\to 0}a_\eps(u) = |u|$ in $L^2_\loc(\R^d)$ and
$\lim_{\eps\to 0}\nabla a_\eps(u) =\nabla |u|$ in $L^2_\loc (\R^d,\R^d)$. Noting that $\frac{u_j}{a_\eps(u)+\eps}$ is uniformly bounded by 1  we see that we can apply the  dominated convergence theorem in the first integral above. For the other two integrals, we apply the monotone convergence theorem, using that $Q$ is strictly elliptic and observing
that  $(a_\eps(u)+\eps)^{-1}$ decreases to  $|u|^{-1}$. Note that in all integrals it is sufficient to integrate over the set $\{u\neq 0\}$. For the first
and third integral, this is obvious due to the presence of $u_j$ which vanish on $\{u=0\}$. For the second one we infer from Stampacchia's lemma that $\nabla u_j = 0$ on $\{u=0\}$.
Thus, by letting $\eps \to 0$, we obtain
\[
\langle\Delta_Q|u|, \varphi \rangle = \int_{\{u\neq 0\}} \sum_{j=1}^m  \Big( \frac{u_j}{|u|}\Delta_Q u_j
+ \frac{1}{u}|\nabla u_j|_Q^2\Big) \varphi \, dx - \int_{\{u\neq 0\}} \frac{1}{|u|} \big|\nabla |u|\big|^2_Q\varphi \, dx.
\]

This proves\eqref{eq.DeltaQ}. Using \eqref{gradient |.| <...} and \eqref{eq.DeltaQ}, also \eqref{eq.kato} follows.
\end{proof}

\section{Generation of semigroup in \texorpdfstring{$L^2(\R^d,\C^m)$}{}}\label{sec. gen of sg in L2}

Let us consider the differential operator $\mathcal{A}u=\Delta_Q u-Vu$, where $u=(u_1,\dots,u_m)$. Here $\Delta_Q$
acts entrywise on $u$, i.e.\ $\Delta_Q u = (\Delta_Q u_1, \ldots, \Delta_Q u_m)$.
 We define $A$ to be the realization of $\mathcal{A}$ on $L^2(\R^d,\C^m)$ with domain
\[ D(A)=\{u\in H^1(\R^d,\C^m) : \mathcal{A}u\in L^2(\R^d,\C^m)\}.\]
In this section we prove that $A$ generates a $C_0$-semigroup of contractions in $L^2(\R^d,\C^m)$.
In view of the Lumer--Phillips theorem, cf. \cite[Theorem II-3.15]{en00}, it suffices to prove that $-A$ is maximal accretive, i.e.\ for $u\in D(A)$
we have $\langle -Au, u\rangle \geq 0$ and $I-A$ is surjective.

To that end, we follow the strategy from \cite{Kato78} and introduce some other realizations of the operator $\mathcal{A}$
on the space $H^{-1}(\R^d; \C^m)$, the dual space of $H^1(\R^d; \C^m)$.
We define the operator $L_0$ by setting
\[
L_0 u=\mathcal{A}u,\qquad u\in D(L_0)\coloneqq C_c^\infty(\R^d;\C^m)
\]
and the operator $L$ by
\[
Lu=\mathcal{A}u,\qquad u\in D(L)\coloneqq \{u\in H^1(\R^d;\C^m) : \mathcal{A}u\in H^{-1}(\R^d;\C^m)\}.\]

We let $\mathcal{\bar{A}}u=\Delta_Q-V^* $  be the formal adjoint of $\mathcal
{A}$, where $V^*$ is the conjugate matrix of $V$. We then define the operators $\bar{L}$ and $\bar{L}_0$ in analogy to
the operators $L$ and $L_0$, using the potential $V^*$ instead of $V$.

We now collect some properties of the operators $L_0$ and $L$ and the adjoints $L_0^*$.
We denote the duality pairing between  $H^{-1}(\R^d;\C^m)$ and  $H^1(\R^d;\C^m)$ by $[\cdot ,\cdot]$.

\begin{prop}\label{adjointe of L_0}
Assume Hypotheses \ref{Hyp}. Then the following hold
\begin{enumerate}
\item[(1)] $\bar{L}=L_0^*$ and $L=\bar{L}_0^*$. Consequently $\bar{L}$ and $L$ are closed.
\item[(2)] $L_0$ is closable and its closure is equal to $L_0^{**}$.
\end{enumerate}
\end{prop}
\begin{proof}
(1) Let $f\in D(\bar{L})$ and $g\in C_c^\infty(\R^d,\C^m)$. Using integration by parts, we see that
\begin{align*}
[\bar{L}f,g]&=\int_{\R^d}\langle \div (Q\nabla f)(x),g(x)\rangle dx-\int_{\R^d}\langle V^*(x)f(x),g(x)\rangle dx\\
&= \int_{\R^d}\langle f(x),\div(Q\nabla g)(x)\rangle dx-\int_{\R^d}\langle f(x),V(x)g(x)\rangle dx\\
&=[L_0 g, f]
\end{align*}
Thus $\bar{L}=L_0^*$ and hence $\bar{L}$ is closed. In a similar way one shows that $L=\bar{L}_0^*$ and thus $L$ is also closed.\\
(2) Since $L_0^*$ is densely defined, $L_0$ is closable with closure $L_0^{**}$ by general theory.
\end{proof}

We can now prove the main result of this section.
\begin{thm}
 The operator $-L$ is maximal monotone.
\end{thm}
\begin{proof}
 \emph{Step 1:}  We first show that $-L_0^{**}$ is maximal monotone.
 It is easy to see that $-L_0$ is monotone. Indeed, for $\varphi \in C_c^\infty (\R^d;\R^m)$ we have
 \[\Re[-L_0\varphi,\varphi]=\int_{\R^d}|\nabla\varphi(x)|_{Q(x)}^2 dx+\Re\int_{\R^d}\langle V(x)\varphi(x),\varphi(x)\rangle dx\ge 0.\]
It follows that also the closure of $-L_0$, i.e.\ the operator $-L_0^{**}$ is monotone.
As $-L_0^{**}$ is monotone, $\rg(1-L^{**}_0)$, the range of $(1-L^{**}_0)$, is a closed subset of $H^{-1}(\R^d;\C^m)$. Therefor, to prove that $-L_0^{**}$ is maximal, it suffices to show that $1-L_0^**$ has dense range. We prove that$(1-L^{**}_0)C_c^\infty(\R^d,\C^m)$ is dense in $H^{-1}(\R^d;\C^m)$.
Since the coefficients of $\mathcal{A}$ are real, it suffices to prove that $(1-L^{**}_0)C_c^\infty(\R^d;\R^m)$ is dense in $H^{-1}(\R^d;\R^m)$. To that end, let $u\in H^1(\R^d;\R^m)$ be such that $[(1-L_0)\varphi,u]=0$ for all $\varphi\in C_c^\infty(\R^d;\R^m)$. Then
\begin{equation}\label{eq.resol-elliptic}
u-\Delta_Q u+V^*u=0,
\end{equation}
and hence,
\[\Delta_Q u_j=u_j+\sum_{l=1}^{m}v_{lj}u_l,\]
for every $j\in\{1,\dots,m\}$ in the sense of distributions. Applying \eqref{eq.kato}, we obtain
\begin{align*}
\Delta_Q|u|&\ge \frac{1}{|u|}\sum_{j=1}^{m}u_j\Delta_Q u_j\chi_{\{u\neq 0\}}\\
&\ge \frac{\chi_{\{u\neq 0\}}}{|u|}(\sum_{j=1}^{m}u_j^2+\sum_{j,l=1}^{m}v_{lj}u_l u_j)\\
&\ge \frac{\chi_{\{u\neq 0\}}}{|u|}|u|^2=|u|.
\end{align*}
Thus, $\Delta |u|\ge |u|$ in the sense of distributions. Now, let $(\phi_n)_n\subset C_c^\infty(\R^d)$ be such that $\phi_n\ge 0$ and $\phi_n\to|u|$ in $H^1(\R^d)$. Then
\begin{align*}
0\le [\Delta_Q|u|,\phi_n]-[|u|,\phi_n]=-\int_{\R^d}\langle\nabla|u|,\nabla\phi_n\rangle_Q dx-\int_{\R^d}|u|\phi_n \,dx.
\end{align*}
Upon $n\to\infty$, we find $-\||\nabla|u||_Q\|_2^2-\|u\|^2_2\ge 0$ which implies that $u=0$. This proves that the range of $I-L_0^{**}$ is dense.\\

\emph{Step 2:} We now prove now that $L=L_0^{**}$. We know that $L$ is a closed extension of $L_0$. Hence $L_0^{**}\subset L$. In order to get the converse inclusion, it suffices to show that $\rho(L)\cap\rho(L_0^{**})\neq\emptyset$. As $L_0^{**}$ is maximal monotone we have $1\in\rho(L_0^{**})$. On the other hand, $(1-L)D(L)\supset(1-L_0^{**})D(L_0^{**})=H^{-1}(\R^d;\C^m)$. Hence, $1-L$ is surjective. To prove that $1-L$ is injective, note that $\ker (1-L)=\rg(1-L^*)^\perp$. Applying \emph{Step 1} with $V^*$ instead of $V$ we see that $-L^*=-\bar{L}_0^{**}$ is maximal. Repeating the above argument, it follows that $\rg (1-L^*) = H^{-1}(\R^d; \C^m)$ and thus $\ker (1-L) = \{0\}$.
This proves that $1\in\rho(L)$ and ends the proof.
\end{proof}

We can now infer that $A$ generates a strongly continuous contraction semigroup.

\begin{coro}
Assume Hypotheses \ref{Hyp}. Then
the operator $A$ generates a $C_0$-semigroup $\{T(t)\}_{t\ge0}$ of contractions on $L^2(\R^d,\C^m)$.
\end{coro}
\begin{proof}
Since $-L$ is monotone, so is $-A$, the part of $-L$ in $L^2(\R^d;\C^m)$. As $-L$ is \emph{maximal} monotone, so
is $-A$. Indeed, given $f \in L^2(\R^d;\C^m) \subset H^{-1}(\R^d;\C^m)$ we find $u\in D(L)$ such that $u-\mathcal{A}u = f$. But then
$\mathcal{A} u = u- f$ belongs to $L^2(\R^d;\C^m)$, proving that $u\in D(A)$ and $u-Au = f$. The claim now follows from the Lumer--Phillips theorem.
\end{proof}

\section{Extension of the semigroup to \texorpdfstring{$L^p(\R^d,\C^m)$}{}}\label{Extension to Lp}

In this section we extrapolate the semigroup $\{T(t)\}_{t\ge0}$ to the spaces $L^p(\R^d,\C^m)$, $1\le p<\infty$. As a first step, we prove that $\{T(t)\}_{t\ge0}$ is given by the Trotter--Kato product formula
\begin{equation}\label{Trotter-Kato}
T(t)f=\lim_{n\to\infty}\left[e^{\frac{t}{n}\Delta_Q}e^{-\frac{t}{n}V}\right]^n f,
\end{equation}
for all $t>0$ and $f\in L^2(\R^d,\C^m)$. Here $\{e^{t\Delta_Q}\}_{t\ge0}$ is the semigroup generated by $\Delta_Q$ in $L^2(\R^d,\C^m)$ and $\{e^{-tV}\}_{t\geq 0}$ is the multiplication semigroup generated by the potential $-V$, i.e.\
$e^{-tV}$ is multiplication with the matrix given pointwise by $\sum_{k=0}^\infty \frac{(-tV(x))^k}{k!}$.
To prove that the semigroup $\{T(t)\}_{t\geq 0}$ is given by the Trotter--Kato formula \eqref{Trotter-Kato} we use the following result which is also of independent interest.

\begin{prop}\label{C_c is a core for A}
Assume Hypotheses \ref{Hyp}. Then $C_c^\infty(\R^d; \C^m)$ is a core for $A$.
\end{prop}
\begin{proof}
Since $-A$ is maximal accretive and has real coefficients, it suffices to show that $(1-A)C_c^\infty(\R^d;\R^m)$ is dense in $L^2(\R^d; \R^m)$. Let $u\in L^2(\R^d;\R^m)$ be such that $\langle(1-A)\varphi,u\rangle=0$, for all $\varphi\in C_c^\infty(\R^d;\R^m)$. Thus, $u-\Delta_Q u+V^* u=0$ in the sense of distributions. Hence,
\[\Delta_Q u_j=u_j+\sum_{l=1}^{m}v_{lj}u_l.\]
In particular, $\Delta_Q u_j=\div(Q\nabla u_j)\in L^2_{\loc}(\R^d)$ for each $j\in\{1,\dots,m\}$.  Then, by local elliptic regularity, see \cite[Theorem 7.1]{Agmon}, $u_j\in H^2_{\loc}(\R^d)$.

Therefore, $|u|=\displaystyle\lim_{\varepsilon\to 0}(|u|^2+\varepsilon^2)^{\frac{1}{2}}$ belongs to $H^2_{\loc}(\R^d)$. In particular, equation \eqref{eq.DeltaQ} still holds true, i.e.
\[ \Delta_Q |u|\ge \frac{\chi_{\{u\neq 0\}}}{|u|}\sum_{j=1}^{m}u_j\Delta_Q u_j\]
almost everywhere. Consequently,
\[\Delta_Q |u|\ge \frac{\chi_{\{u\neq 0\}}}{|u|} (|u|^2+\langle Vu,u\rangle)\ge |u|. \]
 Now, let $\zeta\in C_c^\infty(\R^d)$ be such that $\chi_{B(1)}\le\zeta\le\chi_{B(2)}$ and define $\zeta_n(x)=\zeta(x/n)$ for $x\in\R^d$ and $n\in\N$.
 We multiply both two sides of the inequality $\Delta_Q |u|\ge |u|$ by $\zeta_n |u|$ and integrate by parts. We obtain
 \begin{align*}\label{align pf prop 4.1}
 0&\ge \int_{\R^d}|u(x)|^2\zeta_n(x)dx-\int_{\R^d}\Delta_Q|u|(x)\zeta_n(x)|u(x)|dx\\
 &= \int_{\R^d}|u(x)|^2\zeta_n(x)dx+\int_{\R^d}\langle\nabla(\zeta_n |u|)(x),Q(x)\nabla|u|(x)\rangle dx\nonumber\\
 &= \int_{\R^d}|u(x)|^2\zeta_n(x)dx+\int_{\R^d}|\nabla|u|(x)|_{Q(x)}^2\zeta_n(x)dx+\int_{\R^d}\langle\nabla\zeta_n(x),\nabla|u|(x)\rangle_{Q(x)} |u| dx\nonumber\\
 &\ge \int_{\R^d}|u(x)|^2\zeta_n(x)dx+\frac{1}{2}\int_{\R^d}\langle Q(x)\nabla\zeta_n(x),\nabla|u|^2(x)\rangle dx \nonumber\\
 &= \int_{\R^d}|u(x)|^2\zeta_n(x)dx-\frac{1}{2}\int_{\R^d}\Delta_Q\zeta_n(x)|u(x)|^2dx\nonumber.
 \end{align*}
Here we have used in the fourth line that $\nabla |u|^2 = 2|u|\nabla |u|$.
A straightforward computation shows
\[
\Delta_Q\zeta_n(x) =\frac{1}{n}\sum_{i,j=1}^{m}\partial_i q_{ij}\partial_j\zeta(x/n)+\frac{1}{n^2}\sum_{i,j=1}^{m}q_{ij}\partial_{ij}\zeta(x/n).
\]
It follows that $\|\Delta_Q\zeta_n\|_\infty\to 0$ as $n\to\infty$. Hence, letting $n\to\infty$ in the above  inequality,
we obtain $\|u\|_2\le0$, and thus $u=0$. This finishes the proof.
\end{proof}

\begin{prop}
Assume Hypotheses \ref{Hyp}. Then
the semigroup $\{T(t)\}_{t\ge0}$ is given by the Trotter-Kato product formula \eqref{Trotter-Kato}.
\end{prop}

\begin{proof}
Since $C_c^\infty(\R^d;\C^m)\subset D(\Delta_Q)\cap D(V)$, where $D(\Delta_Q)=H^2(\R^d;\C^m)$ and $D(V)=\{u\in L^2(\R^d;\C^m) : Vu\in L^2(\R^d;\C^m)\}$, the claim follows from \cite[Corollary III-5.8]{Eng-Nag}.
\end{proof}

We can now extend $\{T(t)\}_{t\ge0}$ to $L^p(\R^d,\C^m)$.

\begin{thm}\label{generation-Lp}
Let $1<p<\infty$ and assume Hypotheses \ref{Hyp}. Then $\{T(t)\}_{t\ge0}$ can be extrapolated to a $C_0$-semigroup $\{T_p(t)\}_{t\ge0}$ on $L^p(\R^d,\C^m)$. Moreover, if we denote by $(A_p,D(A_p))$ its generator, then
$A_p u=\mathcal{A}u$, for all $u\in C_c^\infty(\R^d;\C^m)$.
\end{thm}

\begin{proof}
Let $1<p<\infty$ and $f\in L^2(\R^d,\C^m)\cap L^p(\R^d,\C^m)$. Assumption \eqref{eq.diss} yields $|e^{-tV(x)}f(x)|\le|f(x)|$ for all $x\in\R^d$ and $t\ge0$. So $\|e^{-tV}f\|_p\le\|f\|_p$, for all $t\ge 0$.

On the other hand, it is well-known that $\{e^{t\Delta_Q}\}_{t\ge0}$ is a contractive $C_0$-semigroup on $L^p(\R^d,\C^m)$. Consequently, for every $t>0$, both $e^{t\Delta_Q}$ and $e^{-tV}$ leave the set
\[ \mathcal{B}_p\coloneqq\{f\in L^2(\R^d,\C^m)\cap L^p(\R^d,\C^m): \|f\|_p\le1\} \]
invariant. Since $\mathcal{B}_p$ is a closed subset of $L^2(\R^d;\C^m)$ as a consequence of Fatou's lemma, it
follows from the Trotter--Kato formula \ref{Trotter-Kato}, that $T(t)\mathcal{B}_p \subset \mathcal{B}_p$.
It follows that $\|T(t)f\|_p \leq \|f\|_p$ for all $f\in L^2(\R^d;\C^m)\cap L^p(\R^d;\R^m)$. By density, we can extend
$T(t)$ to a contraction $T_p(t)$ on $L^p(\R^d;\C^m)$. The semigroup law for $\{T_p(t)\}_{t\geq 0}$ follows immediately.

Let us prove that $\{T_p(t)\}_{t\geq 0}$ is strongly continuous. To that end, pick $p^* \in (1,\infty)$ and $\theta \in (0,1)$
such that $1/p = (1-\theta)/2+\theta/p^*$. By the interpolation inequality, we find
\[ \|T(t)f-f\|_p\le \|T(t)f-f\|_2^{1-\theta}\|T(t)f-f\|_{p^*}^\theta\le 2^\theta\|T(t)f-f\|_2^{1-\theta}\|f\|_{p^*}^\theta, \]
It follows that $\lim_{t\to 0}T(t)f=f$ in $L^p(\R^d;\C^m)$ for all $f \in C_c^\infty(\R^d,\C^m)$. By density, the strong continuity of the semigroup $\{T_p(t)\}_{t\ge0}$ follows.

Let us now turn to the generator of $\{T_p(t)\}_{t\geq 0}$.
Fix $t>0$ and $f\in C_c^\infty(\R^d;\C^m)$. Then $f\in D(A)$ and
\begin{equation}\label{formule tous les jours}
T(t)f-f=\int_{0}^{t}AT(s)f\,ds=\int_{0}^{t}T(s)\mathcal{A}f\,ds,
\end{equation}
where the integral is computed in $L^2(\R^d;\C^m)$. However, $\mathcal{A}f$ has compact support whence
$\mathcal{A}f\in L^p(\R^d; C^m)$ and the map $t\mapsto T_p(t)\mathcal{A}f$ is  continuous  from $[0,\infty)$ into $L^p(\R^d;\C^m)$. Hence, \eqref{formule tous les jours} holds true in $L^p(\R^d,\C^m)$, i.e.
\[T_p(t)f-f=\int_{0}^{t}T_p(s)\mathcal{A}f\,ds.\]
This implies that $t\mapsto T_p(t)f$ is differentiable in $[0,\infty)$. It follows that
$f\in D(A_p)$ and $A_p f=\mathcal{A}f$.
\end{proof}

\begin{rmk}
It is also possible to extend $T$ to a consistent contraction semigroup $\{T_1(t)\}_{t\geq 0}$ on $L^1(\R^d; \C^m)$.
Mutatis mutandis, the proof is that of \cite[Theorem 3.7]{KLMR}.

\end{rmk}

\section{Maximal domain of \texorpdfstring{$A_p$}{} and further properties}\label{sec. maximal domain}
In this section we characterize the domain $D(A_p)$ of the generator of $\{T_p(t)\}_{t\geq 0}$.
More precisely, we prove that it is the maximal domain in $L^p$.
We first show that the space of test functions is a core for $A_p$.

\begin{prop}
  Let $1<p<\infty$ and assume Hypotheses \ref{Hyp}. Then,
 \begin{itemize}
 \item[(i)] the set of test functions $C_c^\infty(\R^d;\C^m)$ is a core for $A_p$,
 \item[(ii)] the semigroup $\{T_p(t)\}_{t\ge0}$ is given by the Trotter--Kato product formula.
 \end{itemize}
 \end{prop}
 \begin{proof}
 (i) Fix $1<p<\infty$. Since $-A_p$ is m-accretive and the coefficients of $\mathcal{A}$ are real, it suffices to show that $(1-A_p)C_c^\infty(\R^d,\R^m)$ is dense in $L^p(\R^d,\R^m)$. Let $u\in L^{p'}(\R^d;\R^m)$ be such that $\langle(1-\mathcal{A})\varphi,u\rangle_{p,p'}=0$ for all $\varphi\in C_c^\infty(\R^d;\R^m)$. So,
 \begin{equation}\label{ellip res equ}
 u-\Delta_Q u+V^* u=0
 \end{equation}
 in the sense of distributions. In particular, \[\Delta_Q u_j=u_j+\sum_{l=1}^{m}v_{lj}u_l\;\in L^{p'}_\loc(\R^d) \]
for all $j\in\{1,\dots,m\}$. By local elliptic regularity, see \cite[Theorem 7.1]{Agmon}, $u_j\in W^{2,p'}_{\loc}(\R^d)$ for all $j\in\{1,\dots,m\}$. Then, \eqref{ellip res equ} holds true almost everywhere on $\R^d$.

Consider $\zeta\in C_c^\infty(\R^d)$ such that $\chi_{B(1)}\leq\zeta\leq\chi_{B(2)}$ and define $\zeta_n(\cdot)=\zeta(\cdot/n)$ for $n\in\N$. For $p'<2$ we multiply equation \eqref{ellip res equ} by $\zeta_n (|u|^2+\varepsilon^2)^{\frac{p'-2}{2}}u\in L^p(\R^d,\R^m)$ for $\varepsilon>0,\,n\in\N$. Integrating by parts, we obtain
 \begin{eqnarray*}
 0 &=& \int_{\R^d}\zeta_n (|u|^2+\varepsilon^2)^{\frac{p'-2}{2}}|u|^2\, dx+\sum_{j=1}^{m}\int_{\R^d}\Big\langle\nabla u_j,\nabla\big(\zeta_n (|u|^2+\varepsilon^2)^{\frac{p'-2}{2}}u_j\big)\Big\rangle_{Q}\, dx\\
 & & +\int_{\R^d}\zeta_n(|u|^2+\varepsilon^2)^{\frac{p'-2}{2}}\langle V^* u,u\rangle\,dx\\
 &\ge & \int_{\R^d}\zeta_n(|u|^2+\varepsilon^2)^{\frac{p'-2}{2}}|u|^2\,dx+\sum_{j=1}^{m}\int_{\R^d}|\nabla u_j|_{Q(x)}^2 \zeta_n(|u|^2+\varepsilon^2)^{\frac{p'-2}{2}}\,dx\\
 & & +\sum_{j=1}^{m}\int_{\R^d}\langle\nabla u_j,\nabla\zeta_n\rangle_{Q} (|u|^2+\varepsilon^2)^{\frac{p'-2}{2}}u_j\,dx\\
 & & +(p'-2)\sum_{j=1}^{m}\int_{\R^d}\langle\nabla u_j,\nabla|u|\rangle_{Q} u_j|u|\zeta_n(|u|^2+\varepsilon^2)^{\frac{p'-4}{2}}\,dx\\
 &\ge & \int_{\R^d}\zeta_n(|u|^2+\varepsilon^2)^{\frac{p'-2}{2}}|u|^2\, dx+\sum_{j=1}^{m}\int_{\R^d}|\nabla u_j|_{Q(x)}^2 \zeta_n(|u|^2+\varepsilon^2)^{\frac{p'-2}{2}}\,dx\\
  & & +\int_{\R^d}\langle\nabla|u|,\nabla\zeta_n\rangle_{Q}(|u|^2+\varepsilon^2)^{\frac{p'-2}{2}}|u| \,dx\\
 & & +(p'-2)\int_{\R^d}|\nabla|u||_{Q}^2\zeta_n |u|^2(|u|^2+\varepsilon^2)^{\frac{p'-4}{2}}\,dx.\\
 \end{eqnarray*}
 It follows now from \eqref{gradient |.| <...} that
 \begin{eqnarray*}
 0 &\ge & \int_{\R^d}\zeta_n(|u|^2+\varepsilon^2)^{\frac{p'-2}{2}}|u|^2\, dx+\int_{\R^d}\langle\nabla|u|,\nabla\zeta_n\rangle_{Q} (|u|^2+\varepsilon^2)^{\frac{p'-2}{2}}|u| \,dx\\
  & & +(p'-1)\int_{\R^d}|\nabla|u||_{Q}^2\zeta_n |u|^2(|u|^2+\varepsilon^2)^{\frac{p'-4}{2}}\,dx\\
  &\ge & \int_{\R^d}\zeta_n(|u|^2+\varepsilon^2)^{\frac{p'-2}{2}}|u|^2\, dx+\frac{1}{p'}\int_{\R^d}\langle\nabla((|u|^2+\varepsilon^2)^{\frac{p'}{2}})),\nabla\zeta_n\rangle_{Q} \,dx\\
  &=& \int_{\R^d}\zeta_n(|u|^2+\varepsilon^2)^{\frac{p'-2}{2}}|u|^2\, dx-\frac{1}{p'}\int_{\R^d} \Delta_Q\zeta_n (|u|^2+\varepsilon^2)^{\frac{p'}{2}})\,dx.
 \end{eqnarray*}
 Upon $\varepsilon\to 0$, we find
 $$\int_{\R^d}\zeta_n |u|^{p'}\,dx-\frac{1}{p'}\int_{\R^d} \Delta_Q\zeta_n|u|^{p'}\,dx\le 0.$$
As in the proof of Proposition \ref{C_c is a core for A}, upon $n \to \infty$, we conclude that
 $$\int_{\R^d}|u|^{p'}\, dx\le 0.$$
 Therefore, $u=0$.\\
 In the case when $p'>2$, one multiplies in \eqref{ellip res equ} by $\zeta_n|u|^{p'-2}u$ and argues in a similar way.\\
 (ii) This is an immediate consequence of (i) and \cite[Corollary III-5.8]{Eng-Nag}.
\end{proof}
In the next result we show that the domain $D(A_p)$ is equal to the $L^p$-maximal domain of $\mathcal{A}$.
\begin{prop}\label{Coincidence of domains}
Let $1<p<\infty$. Assume Hypotheses \ref{Hyp}. Then
$$D(A_p)=\{u\in L^p(\R^d,\C^m)\cap W^{2,p}_{\loc}(\R^d;\C^m) : \mathcal{A}u\in L^p(\R^d;\C^m)\}:=D_{p,\max}(\mathcal{A}).$$
\end{prop}
\begin{proof}
 Let us show first that $D(A_p)\subseteq D_{p,\max}(\mathcal{A})$. Take $u\in D(A_p)$. Since $C_c^\infty(\R^d,\C^m)$ is a core for $A_p$, it follows that there exists $(u_n)_n\subset C_c^\infty(\R^d,\C^m)$ such that $u_n\to u$ and $\mathcal{A}u_n\to A_p u$ in $L^p(\R^d,\C^m)$, and in particular in $L^p_\loc(\R^d,\C^m)$. As $V\in L^\infty_\loc(\R^d,\C^m)$, we deduce that $V u_n\to Vu$ in $L^p_\loc(\R^d,\C^m)$. Consequently, $$\Delta_Q u=A_pu+Vu=\lim_{n\to\infty}\mathcal{A}u_n+Vu_n \in L^p_\loc(\R^d,\C^m).$$
 So, by local elliptic regularity, we obtain $u\in W^{2,p}_\loc(\R^d,\C^m)$. Hence, $\mathcal{A}u=A_p u$ belongs to $L^p(\R^d,\C^m)$, which shows that $u\in D_{p,\max}(\mathcal{A})$.

In order to prove the other inclusion it suffices to show that $\lambda-\mathcal{A}$ is injective on $D_{p,\max}(\mathcal{A})$, for some $\lambda>0$. To this end, let $u\in D_{p,\max}(\mathcal{A})$ such that $(\lambda-\mathcal{A})u=0$.
Assume that $p\geq 2$. Multiplying by $\zeta_n |u|^{p-2}u$ and integrating (by part) over $\R^d$ one obtains
 \begin{eqnarray*}
 0&=&\lambda\int_{\R^d}\zeta_n(x)|u(x)|^p dx+\int_{\R^d}\sum_{j=1}^{m}\langle Q\nabla u_j,\nabla(|u|^{p-2}u_j\zeta_n)\rangle dx\\
 & & +\int_{\R^d}\langle V(x)u(x),u(x)\rangle |u(x)|^{p-2}\zeta_n(x)dx\\
 &\geq & \lambda\int_{\R^d}\zeta_n(x)|u(x)|^p dx+\int_{\R^d}|u(x)|^{p-2}\zeta_n(x)\sum_{j=1}^{m}\langle Q(x)\nabla u_j(x),\nabla u_j(x)\rangle dx\\
 & & +\int_{\R^d}\sum_{j=1}^{m}|u(x)|^{p-2}u_j(x)\langle Q(x)\nabla u_j(x),\nabla \zeta_n(x)\rangle dx\\
 & & +(p-2)\int_{\R^d}|u(x)|^{p-2}\zeta_n(x)\langle Q(x)\nabla |u|(x),\nabla |u|(x)\rangle dx\\
 &\geq &\lambda\int_{\R^d}\zeta_n(x)|u(x)|^p dx+\int_{\R^d}|u(x)|^{p-1}\langle Q(x)\nabla |u|(x),\nabla \zeta_n(x)\rangle dx\\
 &\geq &\lambda\int_{\R^d}\zeta_n(x)|u(x)|^p dx+\frac{1}{p}\int_{\R^d}\langle Q(x)\nabla \zeta_n(x), \nabla |u|^p(x)\rangle dx\\
 &\geq & \lambda\int_{\R^d}\zeta_n(x)|u(x)|^p dx-\frac{1}{p}\int_{\R^d}\Delta_Q\zeta_n(x) |u(x)|^{p}dx.
 \end{eqnarray*}
 So, as in the proof of the above proposition, we conclude that $u=0$.

 The case $p<2$ can be obtained similarly, by multiplying the equation $(\lambda -\mathcal{A})u=0$ by $\zeta_n(|u|^2+\varepsilon)^{\frac{p-2}{2}}u$, $\varepsilon>0$, instead of $\zeta_n|u|^{p-2}u$.
\end{proof}

We end this article by giving an example which shows that generation of $C_0$-semigroups for scalar-valued Schr\"odinger operators with complex potentials can be deduced from the vector-valued case developed in the previous sections.

\begin{examp}\label{exa-complexe}
 Let us consider the matrix potential
\[ V(x):=\begin{pmatrix}
w(x) & -v(x)\\
v(x) & w(x)
\end{pmatrix}=v(x)\begin{pmatrix}
0 & -1\\
1 & 0
\end{pmatrix}+w(x)\begin{pmatrix}
1 & 0\\
0 & 1
\end{pmatrix}\]
where $v\in L^\infty_{\loc}(\R^d)$ and $0\le w\in L^\infty_{\loc}(\R^d)$. Then Hypotheses \ref{Hyp} are satisfied and we can deduce
from Theorem \ref{generation-Lp} and Proposition \ref{Coincidence of domains}, that $A_p$, the $L^p$-realization of the operator
$$\mathcal{A} =\begin{pmatrix}
\Delta & 0\\
0 & \Delta
\end{pmatrix}-V\, \mbox{\ with domain }\{u\in L^p(\R^d;\C^2)\cap W^{2,p}_{\loc}(\R^d;\C^2) : \mathcal{A}u\in L^p(\R^d,\C^2)\},$$
generates a $C_0$-semigroup on $L^p(\R^d;\C^2)$. Moreover $C_c^\infty(\R^d;\R^2)$ is a core for $A_p$.\\
Diagonalizing the matrix $\begin{pmatrix}
0 & -1\\
1 & 0
\end{pmatrix}$  we see that $A_p$ is similar to a diagonal operator. More precisely, with $P=\begin{pmatrix}
1 & 1\\
-i & i
\end{pmatrix}$
we have
$$P^{-1}L_pP=\begin{pmatrix}
\Delta & 0\\
0 & \Delta
\end{pmatrix}-\begin{pmatrix}
iv+w & 0\\
0 & -iv+w
\end{pmatrix}.$$

It follows that the Schr\"odinger operators  $\Delta \pm iv-w$ with domain
 $$\{f\in L^p(\R^d)\cap W^{2,p}_{\loc}(\R^d,) : \Delta f\pm ivf-wf \in L^p(\R^d)\}$$
 generate $C_0$-semigroups on $L^p(\R^d)$. Moreover $C_c^\infty(\R^d)$ is a core for this operator.
 In general, these semigroups can not be expected to be  not analytic, see \cite[Example 3.5]{KLMR}.
 However, imposing additional assumptions on the potential $V$, e.g.\ that the numerical range is contained in a sector,
 one can also prove analyticity of the semigroup, see \cite[Proposition 4.5]{Mai-Rha}. More precisely, there we find the following result:
 \begin{prop}
 Assume Hypotheses \ref{Hyp} and there is a positive constant $C$ such that
 $$\Re \langle V(x)\xi ,\xi \rangle \ge C|\Im \langle V(x)\xi ,\xi \rangle |,\quad \forall  x\in \R^d,\,\xi \in \C^m.$$ Then the semigroup $\{T_p(t)\}_{t\ge 0}$ can be extended to an analytic semigroup on
 $L^p(\R^d,\C^m)$.
 \end{prop}
 Using this, we see that these semigroups are analytic provided that there is a constant $C>0$ such that $|v(x)|\le Cw(x)$ for a.e.\ $x\in \R^d$.
\end{examp}

\end{document}